\documentclass[10pt]{article}

\usepackage[utf8]{inputenc} 
\usepackage[T1]{fontenc}    

\usepackage{hyperref}       
\hypersetup{
setpagesize=false,
bookmarksnumbered=true,
bookmarksopen=true,
colorlinks=true,
linkcolor=blue,
citecolor=blue,
}

\usepackage{url}            
\usepackage{booktabs}       
\usepackage{amsfonts}       
\usepackage{nicefrac}       
\usepackage{microtype}      
\usepackage{xcolor}         

\usepackage{amsmath,amsfonts,amssymb}
\usepackage{bm}
\usepackage{physics}
\usepackage{graphicx}
\usepackage{algorithm}
\usepackage[noend]{algpseudocode}
\usepackage{tcolorbox}
\tcbuselibrary{breakable,skins,theorems}
\usepackage{amsthm}
\usepackage{breqn}
\usepackage[shortlabels]{enumitem}
\usepackage{subcaption}

\usepackage{cleveref}
\expandafter\def\csname ver@etex.sty\endcsname{3000/12/31}

\usepackage{autonum}
\makeatletter
\autonum@generatePatchedReferenceCSL{Cref}
\makeatother

\usepackage{here}
\usepackage{wrapfig}
\usepackage{booktabs}
\usepackage{multirow}
\usepackage{multicol}
\usepackage{mathtools}
\usepackage{xspace}
\DeclarePairedDelimiter{\rbra}{\lparen}{\rparen} 
\DeclarePairedDelimiter{\sbra}{\lbrack}{\rbrack} 
\DeclarePairedDelimiter{\abra}{\langle}{\rangle} 

\usepackage{thmtools,thm-restate}

\declaretheoremstyle[
shaded={bgcolor=gray!15},
]{thmsty}

\declaretheorem[
name=Theorem,
refname={Theorem,Theorems},
style=thmsty,
]{theorem}

\declaretheorem[
name=Lemma,
refname={Lemma,Lemmas},
style=thmsty,
]{lemma}

\declaretheorem[
name=Assumption,
refname={Assumption,Assumptions},
style=thmsty,
]{assumption}

\crefname{algorithm}{Algorithm}{Algorithms}
\crefname{definition}{Definition}{Definitions}
\crefname{line}{Line}{Lines}
\crefname{lemma}{Lemma}{Lemmas}
\crefname{theorem}{Theorem}{Theorems}
\crefname{section}{Section}{Sections}
\crefname{appendix}{Appendix}{Appendices}
\crefname{table}{Table}{Tables}
\crefname{figure}{Figure}{Figures}
\crefname{equation}{}{}
\crefname{item}{}{}
\Crefname{equation}{Eq.}{Eqs.}
\crefname{assumption}{Assumption}{Assumptions}
\crefname{question}{Question}{Questions}

\setlist[itemize]{
topsep=0.4\baselineskip,
itemsep=0\baselineskip,
leftmargin=1.5em,
}

\newlist{enuminasm}{enumerate}{1} 
\setlist[enuminasm]{
  font=\upshape,
  label=(\alph*),
  ref=\theassumption(\alph*),
  topsep=0.4\baselineskip,
  itemsep=0\baselineskip,
  leftmargin=2em,
}
\crefalias{enuminasmi}{assumption}

\newlist{enuminthm}{enumerate}{1}
\setlist[enuminthm]{
  font=\upshape,
  label=(\alph*),
  ref=\thetheorem(\alph*),
  topsep=0.4\baselineskip,
  itemsep=0\baselineskip,
  leftmargin=2em,
}
\crefalias{enuminthmi}{theorem}

\newcommand{\R}{{\mathbb{R}}}
\newcommand{\euclid}{{\R^d}}
\newcommand{\hessian}{{\nabla^2}}
\newcommand{\inprod}[2]{{\abra*{{#1}, {#2}}}}
\newcommand{\normsq}[1]{{\norm{#1}^2}}
\newcommand{\gnorm}[1]{\norm{\nabla f\rbra*{{#1}}}}
\newcommand{\explain}[1]{{\rbra*{\text{by } {#1}}}}
\newcommand{\since}[1]{\rbra*{\text{since } {#1}}}

\newcommand{\cint}[2]{{\sbra*{{#1},{#2}}}}
\newcommand{\sevenfour}{{O\rbra{\varepsilon^{-7/4}}}}

\newcommand{\sevenfourlog}{{\tilde{O}\rbra{\varepsilon^{-7/4}}}}
\newcommand{\proofsubsection}[1]{\subsection{Proof of \cref{#1}} \label{subsection:proof-of-#1}}

\renewcommand*{\epsilon}{\varepsilon}

\allowdisplaybreaks[4]

\usepackage{pxpgfmark}
\usetikzlibrary{bayesnet}
\usepackage{adjustbox}
\usepackage{empheq}
\usepackage{xparse} 
\usepackage{graphicx}%
\usepackage{multirow}%
\usepackage{amsmath,amssymb,amsfonts}%
\usepackage{amsthm}%
\usepackage{mathrsfs}%
\usepackage[title]{appendix}%
\usepackage{xcolor}%
\usepackage{textcomp}%
\usepackage{manyfoot}%
\usepackage{booktabs}%
\usepackage{algorithm}%
\usepackage{algorithmicx}%
\usepackage{listings}%

\usepackage[margin=1.4truein]{geometry}
\usepackage[numbers,sort&compress]{natbib}
\usepackage{authblk} 
\newcommand{\email}[1]{\href{mailto:#1}{\nolinkurl{#1}}}

\begin{document}
\bibliographystyle{abbrvnat}

\title{Heavy-ball Differential Equation Achieves $O(\varepsilon^{-7/4})$ Convergence for Nonconvex Functions}







\author[1]{Kaito Okamura\footnote{Corresponding author. E-mail: \email{sgwgkaito@g.ecc.u-tokyo.ac.jp}}}
\author[1]{Naoki Marumo}
\author[1,2]{Akiko Takeda}

\affil[1]{Graduate School of Information Science and Technology, University of Tokyo, Tokyo, Japan}
\affil[2]{Center for Advanced Intelligence Project, RIKEN, Tokyo, Japan}

\maketitle

\begin{abstract}
  First-order optimization methods for nonconvex functions with Lipschitz continuous gradient and Hessian have been extensively studied.
	State-of-the-art methods for finding an $\varepsilon$-stationary point within $O(\varepsilon^{-{7/4}})$ or $\tilde{O}(\varepsilon^{-{7/4}})$ gradient evaluations are based on Nesterov's accelerated gradient descent (AGD) or Polyak's heavy-ball (HB) method.
	However, these algorithms employ additional mechanisms, such as restart schemes and negative curvature exploitation, which complicate their behavior and make it challenging to apply them to more advanced settings (e.g., stochastic optimization).
	As a first step in investigating whether a simple algorithm with $O(\varepsilon^{-{7/4}})$ complexity can be constructed without such additional mechanisms, we study the HB differential equation, a continuous-time analogue of the AGD and HB methods.
	We prove that its dynamics attain an $\varepsilon$-stationary point within $O(\varepsilon^{-{7/4}})$ time.
\end{abstract}

\smallskip


\medskip

\section{Introduction}\label{sec:intro}
We consider general nonconvex optimization problems:
\begin{align}
	\min_{x\in \euclid} f(x),\label{eq:nonconvex_optimization}
\end{align}
where $f\colon \euclid \to \R$ is twice differentiable, possibly nonconvex, and lower bounded, i.e., $\inf_{x\in\euclid} f(x) > -\infty$.
\emph{First-order methods}, which use only the information of the function value and gradient of $f$, are widely employed to solve this problem.
First-order methods aimed at finding an $\epsilon$-stationary point (i.e., a point $x \in \euclid$ satisfying $\norm{\nabla f(x)} \leq \epsilon$) with fewer function/ gradient evaluations are actively being researched in the fields of optimization and machine learning~\citep{ghadimi2013stochastic,bottou2018optimization,reddi2016proximal,lei2017nonconvex,li2018simple,zhou2020stochastic,fang2018spider,pham2020proxsarah,wang2019spiderboost,li2021page,NEURIPS2018_d4b2aeb2,pmlr-v75-jin18a,pmlr-v70-carmon17a,xu2017neon+,doi:10.1137/22M1540934,Marumo2024,JMLR:v24:22-0522,li2022restarted}.

A classical result is that the gradient descent (GD) method finds an $\varepsilon$-stationary point within $O(\varepsilon^{-2})$ gradient evaluations under the Lipschitz assumption of $\nabla f$~\citep{nesterov2004introductory}.
It has also been proven that this complexity bound cannot be improved under the same assumption~\citep{Carmon2020}.
From a practical perspective, the advantage of GD lies in the simplicity of the algorithm, which is expressed as a single loop without any conditional branches.
Thus, GD can be extended to address more advanced optimization settings, such as finite-sum and online optimization,
while preserving certain theoretical guarantees.
Many algorithms with complexity bounds, such as SGD~\citep{robbins1951stochastic,ghadimi2013stochastic,bottou2018optimization}, SVRG~\citep{reddi2016proximal}, and others~\citep{lei2017nonconvex,li2018simple,zhou2020stochastic,fang2018spider,pham2020proxsarah,wang2019spiderboost,li2021page}, are designed by building upon the foundation of GD.

Another line of research aims to achieve better complexity bounds than GD by imposing stronger assumptions on the function $f$.
Several first-order methods~\citep{NEURIPS2018_d4b2aeb2,pmlr-v75-jin18a,pmlr-v70-carmon17a,xu2017neon+,doi:10.1137/22M1540934,Marumo2024,JMLR:v24:22-0522,li2022restarted} achieve a complexity bound of $\sevenfour$ or $\sevenfourlog$%
\footnote{%
	The symbol $\tilde{O}\rbra{\cdot}$ hides the polynomial factor of $\log \rbra*{\varepsilon^{-1}}$.
	For example, $\varepsilon^{-2}\rbra{\log\rbra{\varepsilon^{-1}}}^{3} = \tilde O\rbra{\varepsilon^{-2}}$.
}
under the Lipschitz continuity of the Hessian, in addition to the gradient.
These algorithms are based not on GD but on momentum methods, specifically Nesterov's accelerated gradient descent (AGD)~\citep{nesterov1983method} or Polyak's heavy-ball (HB) method~\citep{POLYAK19641}.
In addition, these algorithms incorporate additional mechanisms for improving complexity that are not present in GD-like methods, such as restart schemes and negative curvature exploitation.
These mechanisms introduce conditional branching into the algorithm, complicating its behavior and making it difficult to extend the algorithms and their theoretical guarantees to advanced problem settings.

When considering the aforementioned two lines of research together, a natural question arises: Are additional mechanisms, such as restart schemes, essential to achieve better computational complexity than GD?
More precisely, is it possible to construct a first-order method with the following two properties? 
\begin{itemize}
\item 
  It finds an $\varepsilon$-stationary point within $\sevenfour$ function/gradient evaluations under the Lipschitz assumption on the gradient and Hessian.
\item 
  It is expressed as a single loop without any conditional branches.
\end{itemize}
Considering that some first-order methods~\citep{kim2016optimized,taylor2023optimal,van2018fastest,nesterov1983method} for convex optimization, including AGD, achieve optimal complexity without any conditional branches, it is natural to pursue such a method for nonconvex optimization.
If such a method exists, it would be not only theoretically elegant but also practically desirable, as it would open the door to extending state-of-the-art first-order methods to more advanced settings.

To explore the feasibility of such a first-order method, we analyze the HB ordinary differential equation (HB-ODE)
\begin{align}
	\ddot x(t) & = -\alpha(t) \dot x(t) - \nabla f(x(t)),\label{equation:ode-general}
\end{align}
which serves as a common continuous-time analogy for both AGD and HB methods.
This ODE has been studied mainly for convex cases~\citep{kim2023unifying,krichene2015accelerated,maddison2018hamiltonian,PMID:27834219,su2016differential,wilson2021lyapunov,polyak2017lyapunov,Luo2022,NEURIPS2023_c7074114}.
The advantage of studying ODEs for optimization algorithms is that the behavior of continuous dynamics is often easier to analyze than that of the algorithms themselves.
Leveraging the clarity of ODE analysis via Lyapunov functions, existing algorithms such as AGD have been reinterpreted, and algorithms with new features have been proposed within the context of convex optimization~\citep{kim2023unifying,krichene2015accelerated,maddison2018hamiltonian,PMID:27834219,su2016differential,wilson2021lyapunov,polyak2017lyapunov,Luo2022,NEURIPS2023_c7074114}.
In this study, as an initial step toward confirming the feasibility of the above-mentioned research question, we examine the convergence rate of the average solution to the HB-ODE \eqref{equation:ode-general} to an $\epsilon$-stationary point.

\paragraph{Our contribution.}\label{paragraph:contribution}
We analyze the HB-ODE \eqref{equation:ode-general} with a specific choice of $\alpha(t)$ and prove that, under the Lipschitz assumptions on the gradient and Hessian, a point derived from the solution trajectory of the HB-ODE reaches an $\epsilon$-stationary point within $\sevenfour$ time.
Since the ODE does not contain any complicated mechanisms, such as restart schemes or negative curvature exploitation, this suggests the possibility of constructing a simpler algorithm with the same complexity.
To the best of our knowledge, this is the first analysis of the HB-ODE
in nonconvex optimization that establishes a convergence rate matching the best-known complexity bound of first-order methods under the same assumptions.

The HB-ODE \eqref{equation:ode-general} involves a friction parameter $\alpha(t)$, 
and its appropriate setting for achieving fast convergence has been studied for convex optimization~\citep{su2016differential,wilson2021lyapunov,suh2022continuous,NEURIPS2023_c7074114}.
For instance, \citet{su2016differential} analyzed the case $\alpha(t) = 3/t$ and reassessed the efficiency of AGD.
We propose setting $\alpha(t) = \Theta(T^{-1/7})$, where $\alpha(t)$ depends on the termination time $T$ of the ODE.
This somewhat unconventional setting of the friction parameter $\alpha(t)$ is derived from our ODE analysis and, to the best of our knowledge, is the first of its kind.

A typical approach in the convergence analysis of ODEs is to construct a suitable Lyapunov function~\citep{kim2023unifying,krichene2015accelerated,maddison2018hamiltonian,PMID:27834219,su2016differential,wilson2021lyapunov,polyak2017lyapunov,Luo2022,NEURIPS2023_c7074114}.
This paper employs a different approach; we analyze an average trajectory and evaluate gradients along it.

\paragraph{Notation.}\label{paragraph:notation}
For $x, y \in \euclid$, let $\inprod{x}{y}$ denote the inner product of $x$ and $y$ and $\norm{x}$ denote the Euclidean norm.
For a matrix $A \in \R^{m \times n}$, let $\norm{A}$ denote the operator norm of $A$, which is equal to the largest singular value of $A$.
We write the derivative of $x\colon \R\to\euclid$ as $\dot{x}(t) \coloneqq \dv{x(t)}{t}$.




\section{Related Work}
\label{section:related-work}

\paragraph{Analysis of HB-ODE for convex functions.}
The HB-ODE~\cref{equation:ode-general} was first introduced by \citet{POLYAK19641}, with early studies on its convergence found in \citep{attouch2000heavy,attouch2000heavy2,alvarez2001inertial}.
\citet{su2016differential} proved that the objective function value converges at a rate of $O(t^{-2})$ for the ODE with $\alpha(t) = 3/t$.
The same rate was shown in a more general framework using Bregman divergence by \citet{krichene2015accelerated}.
This rate was later improved to $o(1/t^2)$ by \citet{attouch2016rate} and \citet{may2017asymptotic}.
\citet{kim2023unifying} established a unified framework encompassing both convex and strongly convex settings, with the rate $O ( \min \{ 1/t^2, e^{-\sqrt{\mu} t} \} )$, where $\mu \geq 0$ is the strong convexity parameter.
Extensions to monotone inclusion problems and differential inclusions were also developed, e.g., in~\citep{bot2016second}.

\paragraph{Analysis of HB-ODE for nonconvex functions.}
Compared to convex settings, the analysis of the HB-ODE for nonconvex functions remains relatively limited.
\citet{goudou2009gradient} proved the convergence of the solution under the assumption that $f$ is quasi-convex.
\citet{aujol2022convergence} showed an $O(e^{-\sqrt{2\mu} t})$ convergence rate of the objective value under $\mu$-strong quasi-convexity.
The convergence rate under a general growth condition has also been investigated~\citep{aujol2023convergence}.
\citet{Le2024} studied a differential inclusion generalizing the ODE~\cref{equation:ode-general} to establish the convergence of a stochastic HB method for nonsmooth and nonconvex functions.
In this paper, we analyze the HB-ODE under a different set of assumptions from the above studies, namely the Lipschitz continuity of the gradient and Hessian.
To the best of our knowledge, this is the first analysis of the nonconvex HB-ODE under such relatively mild conditions.

\paragraph{First-order methods with complexity bounds of $\sevenfour$ or $\sevenfourlog$.}
\cref{table:first-order-methods} summarizes existing first-order methods achieving complexity bounds of $\sevenfour$ or $\sevenfourlog$ under the Lipschitz continuity of the gradient and Hessian.
\citet{pmlr-v70-carmon17a} first developed a first-order method with $\sevenfourlog$ complexity. The algorithm is relatively complicated, as it iteratively solves regularized subproblems with AGD and exploits negative curvature directions of $f$.
Subsequently, several other algorithms achieving the same complexity have been proposed~\citep{NEURIPS2018_d4b2aeb2,xu2017neon+,pmlr-v75-jin18a}.
\citet{li2022restarted} improved the complexity to $\sevenfour$, and several other methods have also achieved the same complexity~\citep{doi:10.1137/22M1540934,Marumo2024,JMLR:v24:22-0522}.
An important motivation behind the above research is not only to improve complexity but also to simplify the algorithms.
In fact, more recent algorithms are simpler and more practical.
The methods in~\citep{li2022restarted,doi:10.1137/22M1540934,Marumo2024,JMLR:v24:22-0522} achieve superior complexity bounds using restart schemes that are relatively simple compared to solving subproblems iteratively or exploiting negative curvature.

\begin{table}[t]
	\caption{
		First-order methods under the Lipschitz assumption on the gradient and Hessian.
	}
	\label{table:first-order-methods}
	\centering
	\smallskip
	\begin{tabular}{@{}cccl@{}}\toprule
		                                             & Complexity         & Based on & Additional mechanisms                      \\
		\midrule
		\citep{NEURIPS2018_d4b2aeb2,xu2017neon+}     & $\sevenfourlog$ & AGD      & Subproblem, Negative curvature, Randomness \\
		\citep{pmlr-v70-carmon17a}                   & $\sevenfourlog$ & AGD      & Subproblem, Negative curvature             \\
		\citep{pmlr-v75-jin18a}                      & $\sevenfourlog$ & AGD      & Negative curvature, Randomness             \\
		\citep{doi:10.1137/22M1540934,li2022restarted}  & $\sevenfour$    & AGD      & Restart                                    \\
		\citep{Marumo2024,JMLR:v24:22-0522} & $\sevenfour$    & HB       & Restart                                    \\
		\bottomrule
	\end{tabular}
\end{table}

\paragraph{Heavy-ball method.}
Polyak's HB method~\citep{POLYAK19641} and its variants~\citep{sutskever2013importance,kingma2015adam,reddi2018convergence,cutkosky2020momentum} are widely applied to nonconvex optimization problems in machine learning.
Despite its great practical success, research on the theoretical performance of the HB method without restart schemes remains limited.
\citet{oneill2019behavior} showed that the original HB method is unlikely to converge to strict saddle points.
\citet{ochs2014ipiano} introduced a proximal variant of HB with a complexity bound of $O(\epsilon^{-2})$ under the Lipschitz assumption for gradients.

\section{Analysis of Heavy-ball ODE}\label{section:continuous}
We focus on the following ordinary differential equation (ODE):
\begin{equation}
	\ddot x(t) = -\alpha \dot x(t) - \nabla f(x(t)),\quad
	x(0) = x_0,\quad
	\dot x(0) = 0,\label{equation:ode}
\end{equation}
where $\alpha > 0$ is a fixed parameter.
This ODE can be interpreted as the equation of motion of a particle subject to a conservative force $- \nabla f(x(t))$ and a frictional force $-\alpha \dot x(t)$.
Discretizing this ODE leads to first-order methods such as the AGD and HB methods; analyzing the ODE will provide insights into their design.

We analyze the ODE~\cref{equation:ode} under the following Lipschitz assumptions:
\begin{assumption}\label{assumption:L1L2}
	For some $L_1, L_2 > 0$,
	\begin{enuminasm}
		\item\label{assumption:L1}
		$\norm{\nabla f(y) - \nabla f(x)} \le L_1 \norm{y-x}$ for all $x, y\in \euclid$,
		\item\label{assumption:L2}
		$\norm{\hessian f(y) - \hessian f(x)} \le L_2 \norm{y-x}$ for all $x, y\in \euclid$.
	\end{enuminasm}
\end{assumption}

\subsection{Uniqueness of the Solution}
Let us first discuss the uniqueness of the solution to the ODE~\cref{equation:ode}.
The ODE is equivalent to the following first-order ODE:
\begin{align}
	\dv{t} \begin{pmatrix}
					 x(t) \\
					 \dot x(t)
				 \end{pmatrix}
	= \begin{pmatrix}
			\dot x(t) \\
			-\alpha\dot x(t)-\nabla f(x(t))
		\end{pmatrix}
	\label{equation:ode-first-order}
\end{align}
Under \cref{assumption:L1}, the mapping
$\begin{pmatrix}
		v_1 \\
		v_2
	\end{pmatrix} \mapsto \begin{pmatrix}
		v_2 \\
		-\alpha v_2 - \nabla f(v_1)
	\end{pmatrix}$
is also Lipschitz continuous, and thus there exists a unique solution to the ODE \cref{equation:ode-first-order}.

\subsection{Convergence Rate}
To achieve fast convergence, we consider the average trajectory of the ODE solution:
\begin{equation}
	{\bar{x}}(t) \coloneqq \int_{0}^{t} w_t(s) x(s) \dd s \label{equation:bar-x-t-def}
\end{equation}
for $t > 0$ and $\bar x(0) \coloneqq x(0)$, where $w_t \colon [0, t] \to [0, \infty)$ is a weight function such that $\int_{0}^{t} w_t(s) \dd s = 1$.
Such average solutions have also been studied in the analysis of existing first-order methods~\citep{doi:10.1137/22M1540934,li2022restarted,Marumo2024,JMLR:v24:22-0522} based on AGD or HB for nonconvex optimization.
In this paper, we define the weight in a somewhat subtle manner as
\begin{equation}
	w_t(s)
	\coloneqq
	\frac{\alpha e^{\alpha s}}{e^{\alpha t} - 1}.
	\label{eq:def_w}
\end{equation}
Note that $\alpha$ in this definition corresponds to the friction coefficient in the ODE~\cref{equation:ode}.
This choice of the weight function plays an essential role in our convergence analysis.

The main result of this section is the following theorem.
It should be noted that this theorem evaluates the norm of the gradients at $\bar x(t)$, rather than at $x(t)$ itself.
\begin{theorem}\label{theorem:continuous-convergence-rate}
	Suppose that \cref{assumption:L1L2} holds, and let ${\Delta_f} \coloneqq f(x_0) - \inf_{x\in\euclid} f(x)$.
	Fix $T > 0$ arbitrarily, and set $\alpha$ in the ODE~\cref{equation:ode} as
	\begin{equation}
		\alpha = \rbra{3 L_2}^{\frac{2}{7}} \rbra*{\frac{\Delta_f}{T}}^{\frac{1}{7}}.
		\label{eq:def_alpha}
	\end{equation}
	Then, the following holds:
	\begin{equation}
		\min_{0 \le t \le T} \norm{\nabla f(\bar x{(t)})} \le \frac{7}{6} (3 L_2)^{\frac{1}{7}} \rbra*{\frac{\Delta_f}{T}}^{\frac{4}{7}}
		+ O\rbra*{T^{-\frac{10}{7}}}.
		\label{inequation:theorem-x-bar-norm-upper-cts}
	\end{equation}
\end{theorem}

\cref{theorem:continuous-convergence-rate} states that $\min_{0 \le t \le T} \gnorm{\bar x(t)} = O\rbra{T^{-{4}/{7}}}$, which implies that $\norm{\nabla f\rbra*{\bar x(t)}} \leq \varepsilon$ holds for some $t = \sevenfour$.
This convergence rate is consistent with the state-of-the-art complexity bound of $\sevenfour$ for the first-order methods discussed in \cref{section:related-work}.
\cref{theorem:continuous-convergence-rate} suggests the possibility of developing a first-order method with a complexity bound of $\sevenfour$ without any additional mechanisms by discretizing the ODE \cref{equation:ode}.

Let us provide some additional remarks on \cref{theorem:continuous-convergence-rate}.
The Lipschitz constant $L_1$ of $\nabla f$ does not appear on the right-hand side of \cref{inequation:theorem-x-bar-norm-upper-cts}; the Lipschitz continuity is used only to guarantee the uniqueness of the solution.
Note that the trajectories $x(t)$ and $\bar{x}(t)$ depend on the fixed parameter $T$ because the parameter $\alpha$ is defined as a function of $T$ in \cref{theorem:continuous-convergence-rate}.
Differential equations whose solution trajectories depend on the termination time are sometimes used in the analysis of optimization methods~\citep{suh2022continuous,kim2023unifying}.

\proofsubsection{theorem:continuous-convergence-rate}
\label{sec:proof-sketch-continuous}
This section shows the proof of \cref{theorem:continuous-convergence-rate} and some key ideas for it.
Below is the first key lemma.
\begin{lemma} \label{lemma:gradient-mean-continuous}
	Let $t > 0$ and $z\colon {\cint{0}{t}} \rightarrow \euclid$.
	Suppose \cref{assumption:L2} holds and that a function $w\colon {\cint{0}{t}}\rightarrow [0, \infty)$ satisfies $\int_{0}^{t} w(s) \dd s = 1$. Let $\bar z\coloneqq \int_{0}^{t} w(s) z(s) \dd s$.
	Then, the following holds:
	\begin{align}
		\norm{\nabla f(\bar z) - \int_{0}^{t} w(s) \nabla f(z(s)) \dd s}
		\le \frac{L_2}{2}
		\int_{0}^t {
		\normsq{\dot z(s)}
		}
		\rbra*{{\int_{0}^s \dd \sigma \int_s^t \dd \tau \ w(\sigma) w(\tau)
					(\tau - \sigma) }}
		\dd s.
	\end{align}
\end{lemma}

\begin{proof}
  We now proceed to show each of the following equalities and inequalities:
	\begin{align}
		\norm{
			\nabla f(\bar z) - \int_{0}^{t} w(s)\nabla f(z(s)) \dd s
		}
		&\le
		\frac{L_2}{2} \int_{0}^{t} w(s)\normsq{z (s) - \bar z } \dd s
		\label{inequation:lemma-continuous-first-appendix}\\
		&=
		\frac{L_2}{2} \iint_{0 \le \sigma \le \tau \le t} w(\sigma) w(\tau)\normsq{z(\tau) - z(\sigma)} \dd \sigma \dd \tau
		\label{inequation:lemma-continuous-second-appendix}\\
		&\le
		\frac{L_2}{2}
		\int_{0}^{t} \normsq{\dot z(s)}
		\rbra*{{\int_{0}^{s} \dd \sigma \int_{s}^{t} \dd \tau \ w(\sigma) w(\tau)
					(\tau - \sigma) }}
		\dd s.
		\label{inequation:lemma-continuous-third-appendix}
	\end{align}
	For $s\in \cint{0}{t}$, we have
	\begin{align}
		\nabla f(z(s)) - \nabla f(\bar z)
		 & = \int_0^1 \hessian f((1 - \sigma)\bar z + \sigma z(s)) (z (s) - \bar z ) \dd \sigma
		\\ & = \hessian f(\bar z)(z(s) - \bar z)
		\\ & \hspace{18pt} + \int_0^1 \rbra*{
		\hessian f((1 - \sigma)\bar z + \sigma z(s))  - \hessian f(\bar z)
		}
		(z (s) - \bar z ) \dd \sigma.
		\label{equation:diff-nabla}
	\end{align}
	By multiplying \cref{equation:diff-nabla} by $w(s)$ and integrating it on $s\in \cint{0}{t}$, we have
	\begin{align}
		 & \int_{0}^{t} w(s) \rbra*{\nabla f(z(s)) - \nabla f(\bar z)} \dd s
		\\ & \hspace{30pt} = \int_{0}^{t} w(s) \rbra*{\hessian f(\bar z)(z(s) - \bar z)} \dd s
		\\ & \hspace{60pt} + \int_{0}^{t} w(s) {\int_0^1 \rbra*{
			\hessian f((1 - \sigma)\bar z + \sigma z(s))  - \hessian f(\bar z)
		}
		(z (s) - \bar z ) \dd \sigma} \dd s.
		\label{equation:diff-nabla-integral}
	\end{align}
	The left-hand side of \cref{equation:diff-nabla-integral} is
	\begin{align}
		\int_{0}^{t} w(s) \rbra*{\nabla f(z(s)) - \nabla f(\bar z)} \dd s
		 & =
		\int_0^{t} w(s) \nabla f(z(s)) \dd s
		- \nabla f(\bar z)
		\label{equation:diff-nabla-lhs}
	\end{align}
	since $\int_0^{t} w(s) \dd s = 1$.
	The first term on the right-hand side of \cref{equation:diff-nabla-integral} is
	\begin{align}
		\int_{0}^{t} w(s) \rbra*{\hessian f(\bar z)(z(s) - \bar z)}\dd s
		 & = \hessian f(\bar z) \int_{0}^{t} w(s) (z(s) - \bar z)\dd s
		\\ & = \hessian f(\bar z) \rbra*{\int_{0}^{t} w(s) z(s) \dd s - \rbra*{\int_{0}^{t}w(s) \dd s} \bar z}
		\\ & = \hessian f(\bar z) \rbra*{\bar z - \bar z}
		\\ & = 0 \label{equation:diff-nabla-rhs-first}.
	\end{align}
	Furthermore, the norm of the second term on the right-hand side of \cref{equation:diff-nabla-integral} is evaluated as follows:
	\begin{align}
		 & \phantom{\le\ } \norm{\int_{0}^{t} w(s)
		{ \int_0^1 \rbra*{
					\hessian f((1 - \sigma)\bar z + \sigma z(s))  - \hessian f(\bar z)
				}
				(z (s) - \bar z ) \dd \sigma
				} \dd s}
		\\ & \le \int_{0}^{t} w(s)
		{ \int_0^1
				\norm{\hessian f((1 - \sigma)\bar z + \sigma z(s))-\hessian f(\bar z)}
				\norm{z (s) - \bar z} \dd \sigma
			} \dd s
		\\ & \le \int_{0}^{t} w(s)
		{ \int_0^1
				L_2 \norm{((1 - \sigma)\bar z + \sigma z(s))-\bar z}
				\norm{z (s) - \bar z} \dd \sigma
			} \dd s
		\hspace{18pt}
		\explain{\text{\cref{assumption:L2}}}
		\\ & = \int_{0}^{t} w(s)
		\int_0^1
		\sigma L_2\normsq{z (s) - \bar z} \dd \sigma
		\dd s
		\\ & = \frac{L_2}{2} \int_{0}^{t} w(s)\normsq{z (s) - \bar z } \dd s
		\label{equation:diff-nabla-rhs-second}.
	\end{align}
	By taking norm of both sides of \cref{equation:diff-nabla-integral} and applying \cref{equation:diff-nabla-lhs,equation:diff-nabla-rhs-first,equation:diff-nabla-rhs-second}, we have
	\begin{align}
		\norm{\int_0^{t} w(s) \nabla f(z(s)) \dd s
			- \nabla f(\bar z) } \le \frac{L_2}{2} \int_{0}^{t} w(s)\normsq{z (s) - \bar z } \dd s,
	\end{align}
	which is equivalent to \cref{inequation:lemma-continuous-first-appendix}.

        Next, we derive \cref{inequation:lemma-continuous-second-appendix} from \cref{inequation:lemma-continuous-first-appendix}.
        Since
	\begin{align}
		\normsq{\int_{0}^{t} w(\sigma) z(\sigma) \dd \sigma}
		 & = \inprod{\int_{0}^{t} w(\sigma) z(\sigma) \dd \sigma}{\int_{0}^{t} w(\tau) z(\tau) \dd \tau}
		\\ & = \iint_{{\cint{0}{t}}^2} w(\sigma) w(\tau) \inprod{z(\sigma)}{z(\tau)} \dd \sigma \dd \tau,
		\label{lemma-continuous-sublem1}
	\end{align}
	we have
	\begin{align}
		 & \iint_{0 \le \sigma \le \tau \le t} w(\sigma) w(\tau)\normsq{z(\tau) - z(\sigma)} \dd \sigma \dd \tau
		\\ &= \frac{1}{2} \iint_{{\cint{0}{t}}^2} w(\sigma) w(\tau) \normsq{z(\tau) - z(\sigma)} \dd \sigma \dd \tau
		\\ &=
		\int_{0}^{t} w(\sigma) \normsq{z(\sigma)} \dd \sigma
		- \iint_{{\cint{0}{t}}^2} w(\sigma) w(\tau) \inprod{z(\sigma)}{z(\tau)} \dd \sigma \dd \tau
		\\ &=
		\int_{0}^{t} w(\sigma) \normsq{z(\sigma)} \dd \sigma
		- \normsq{\int_{0}^{t} w(\sigma) z(\sigma) \dd \sigma}
		\quad \explain{\cref{lemma-continuous-sublem1}}
		\\ &=
		\int_{0}^{t} w(\sigma) \normsq{z(\sigma)} \dd \sigma
		- \normsq{\bar z}
		\\ &= \int_{0}^{t} w(\sigma) \normsq{z(\sigma)} \dd \sigma
		- 2 \normsq{\bar z} + \normsq{\bar z}
		\\ &=
		\int_{0}^{t} w(\sigma) \normsq{z(\sigma)} \dd \sigma
		- 2 \inprod{\int_{0}^{t} w(\sigma)z(\sigma)\dd \sigma}{\bar z} + \normsq{\bar z} \int_{0}^{t} w(\sigma)\dd \sigma
		\\ & \phantom{daaaaaaaaaaaaaaaaaaaa} \since{\bar z\coloneqq \int_{0}^{t} w(\sigma) z(\sigma) \dd \sigma \text{ and } \int_{0}^{t} w(\sigma) \dd \sigma = 1}
		\\ &=
		\int_{0}^{t} w(\sigma) \normsq{z(\sigma) - \bar z} \dd \sigma,
	\end{align}
	which is equivalent to \cref{inequation:lemma-continuous-second-appendix}.

	Next we prove inequality \cref{inequation:lemma-continuous-third-appendix}.
	For $0\le \sigma \le \tau \le t$, we have
	\begin{alignat}{2}
		\normsq{z(\tau) - z(\sigma)}
				 & = \normsq{\int_{\sigma}^{\tau} \dot z(s) \dd s}
		\\ & \le \rbra*{{
					\int_{\sigma}^{\tau} \norm{\dot z(s)} \dd s
		}}^2 & \quad                                           & \explain{\text{the triangle inequality}}
		\\ & \le \rbra*{{
					\int_{\sigma}^{\tau} \normsq{\dot z(s)} \dd s
				}} \rbra*{\int_{\sigma}^{\tau} \dd s}
				 & \quad                                           & \explain{\text{the Cauchy--Schwarz inequality}}
		\\ & = (\tau - \sigma) {
				\int_{\sigma}^{\tau} \normsq{\dot z(s)} \dd s
			} \label{inequation:cs-integral}.
	\end{alignat}
	Using this, we also have
	\begin{align}
		\int_{0 \le \sigma \le \tau \le t} w(\sigma)& w(\tau)\normsq{z(\tau) - z(\sigma)} \dd \sigma \dd \tau
		\\ & \le \int_{0 \le \sigma \le \tau \le t} w(\sigma) w(\tau)
		(\tau - \sigma) \int_{\sigma}^{\tau} \normsq{\dot z(s)} \dd s
		\dd \sigma \dd \tau
		\\ & = \int_{0 \le \sigma \le s \le \tau \le t} w(\sigma) w(\tau)
		(\tau - \sigma) {
				\normsq{\dot z(s)} \dd s
			}
		\dd \sigma \dd \tau
		\\ & =
		\int_{0}^{t} \normsq{\dot z(s)}
		\rbra*{{\int_{0}^{s} \dd \sigma \int_{s}^{t} \dd \tau \ w(\sigma) w(\tau)
					(\tau - \sigma) }}
		\dd s,
	\end{align}
	which implies \cref{inequation:lemma-continuous-third-appendix}.
\end{proof}

\cref{lemma:gradient-mean-continuous} bounds the error when the gradient $\nabla f(\bar z)$ at the average solution $\bar z$ is approximated by the average of the gradients $\int_{0}^{t} w(s) \nabla f(z(s)) \dd s$.
Applying this lemma with $(z, w) = (x, w_t)$, where $x$ is the solution to the ODE \cref{equation:ode} and $w_t$ is defined by \cref{eq:def_w}, gives the following bound:
\begin{align}
	\norm{\nabla f(\bar x(t)) - \int_{0}^{t} w_t(s) \nabla f(x(s)) \dd s}
	\le \frac{L_2}{2}
	\int_{0}^t {
	\normsq{\dot x(s)}
	}
	\rbra*{{\int_{0}^s \dd \sigma \int_s^t \dd \tau \ w_t(\sigma) w_t(\tau)
				(\tau - \sigma) }}
	\dd s.
	\label{eq:gradient-mean-continuous2}
\end{align}
Thanks to the definition \cref{eq:def_w} of the weight $w_t$, the average of the gradients simplifies to
\begin{alignat}{2}
	\int_{0}^{t} w_t(s) \nabla f(x(s)) \dd s
	 & = -\int_{0}^{t} w_t(s) \rbra*{
		\ddot{x}(s) + \alpha\dot{x}(s)
	} \dd s
	 & \quad                          & \explain{\cref{equation:ode}}
	\\ &= -\int_{0}^{t} \rbra*{w_t(s)\ddot{x}(s) + \dot w_t(s)\dot{x}(s)} \dd s
	 & \quad                          & \explain{\cref{eq:def_w}}
	\\ &= -\sbra*{
		w_t(s) \dot{x}(s)
	}_{s=0}^{t}
	\\ &= - w_t(t) \dot x(t)
	\\ & = - \frac{\alpha}{1 - e^{- \alpha t}} \dot x(t).
	 & \quad                          & \explain{\cref{eq:def_w}}
	 \label{eq:weighted-average-of-gradient-simple-continuous}
\end{alignat}
Substituting this equation into the left-hand side of \cref{eq:gradient-mean-continuous2} and doing some calculations, we can obtain the following upper bound on $\norm{\nabla f({\bar{x}(t)})}$.
\begin{lemma}
	\label{lemma:gnorm-bar-x-continuous}
	The following holds for all $t > 0$:
	\begin{align}
		\norm{\nabla f({\bar{x}(t)})}
		\le
		\frac{\alpha}{1 - e^{- \alpha t}} \norm{\dot x(t)}
		+
		\frac{L_2}{2}
		\frac{1}{\rbra*{1 - e^{- \alpha t}}^2}
		\int_{0}^{t} \normsq{\dot x(s)}
		e^{- \alpha (t - s)} \rbra*{t - s}
		\dd s.
	\end{align}
\end{lemma}

\begin{proof}
	Applying \cref{lemma:gradient-mean-continuous} with $(z, w) = (x, w_t)$, where $x$ is the solution to the ODE~\cref{equation:ode} and $w_t$ is defined by \cref{eq:def_w}, gives the following bound:
	\begin{align}
		\norm{\nabla f(\bar x(t)) - \int_{0}^{t} w_t(s) \nabla f(x(s)) \dd s}
		\le \frac{L_2}{2}
		\int_{0}^t {
		\normsq{\dot x(s)}
		}
		\rbra*{{\int_{0}^s \dd \sigma \int_s^t \dd \tau \ w_t(\sigma) w_t(\tau)
					(\tau - \sigma) }}
		\dd s.
	\end{align}
	Using the triangle inequality gives
	\begin{align}
		\norm{\nabla f(\bar x(t))}
		& \le
		\norm{\int_{0}^{t} w_t(s) \nabla f(x(s)) \dd s} \\
		& \hspace{10pt} + \frac{L_2}{2}
		\int_{0}^t {
		\normsq{\dot x(s)}
		}
		\rbra*{\int_{0}^s \dd \sigma \int_s^t \dd \tau \ w_t(\sigma) w_t(\tau)
			(\tau - \sigma) }
		\dd s.
	\end{align}
	Now, we will evaluate each term on the right-hand side.
	The first term is evaluated using \cref{eq:weighted-average-of-gradient-simple-continuous} as follows:
	\begin{align}
		\norm{\int_{0}^{t} w_t(s) \nabla f(x(s)) \dd s}
		= \frac{\alpha}{1 - e^{- \alpha t}} \norm{\dot x(t)}.
	\end{align}
	The second term is evaluated as follows:
	\begin{align}
		\int_{0}^s \dd \sigma \int_s^t \dd \tau \ w_t(\sigma) w_t(\tau) (\tau - \sigma)
		 & =
		\rbra*{\frac{\alpha}{e^{\alpha t} - 1}}^2
		{\int_{s}^{t} e^{\alpha \tau} \rbra*{\int_{0}^{s} e^{\alpha \sigma} (\tau-\sigma) \dd \sigma} \dd \tau }
		\\ & =
		\frac{\alpha}{(e^{\alpha t} - 1)^2} \int_{s}^{t}
		e^{\alpha \tau} \rbra*{ \sbra*{
				e^{\alpha \sigma} \rbra*{
					\tau - \sigma + \frac{1}{\alpha}
				}
			}_{\sigma=0}^s}
		\dd \tau
		\\ &= \frac{\alpha}{(e^{\alpha t} - 1)^2} \int_{s}^{t} e^{\alpha \tau} \rbra*{
			e^{\alpha s} \rbra*{
				\tau - s + \frac{1}{\alpha}
			} - \rbra*{
				\tau + \frac{1}{\alpha}
			}
		} \dd \tau
		\\ &\le \frac{\alpha}{(e^{\alpha t} - 1)^2} \int_{s}^{t} e^{\alpha \tau}
		e^{\alpha s} \rbra*{
			\tau - s + \frac{1}{\alpha}
		} \dd \tau
		\\ &= \frac{e^{\alpha s}}{(e^{\alpha t} - 1)^2} \sbra*{
			e^{\alpha \tau} \rbra*{
				\tau - s
			}
		}_{\tau=s}^t
		\\ &= \frac{e^{\alpha (s + t)}}{\rbra*{e^{\alpha t}-1}^2}
		\rbra*{t - s}
		\\ &= \frac{1}{\rbra*{1 - e^{- \alpha t}}^2}
		e^{- \alpha (t - s)} \rbra*{t - s},
	\end{align}
	which completes the proof.
\end{proof}


The following lemma follows from the relationship between mechanical energy and the work done by friction in the ODE \cref{equation:ode}.
\begin{lemma}
	\label{lemma:normsq_xdot_upperbound}
	The following holds for all $t \ge 0$:
	\begin{align}
		\int_{0}^{t} \normsq{\dot x(s)} \dd s
		\le
		\frac{\Delta_f}{\alpha},
		\label{equation:function-defcrease-continuous}
	\end{align}
        where $\Delta_f$ is defined in \Cref{theorem:continuous-convergence-rate}.
\end{lemma}
\begin{proof}
	Let
	\begin{align}
		\Phi(t) \coloneqq \frac{1}{2} \normsq{\dot x(t)} + f(x(t)).
		\label{equation:Phi}
	\end{align}
	Differentiating both sides of \cref{equation:Phi} by $t$, we obtain
	\begin{align}
		\dv{t} \Phi(t) & = \inprod{\dot x(t)}{\ddot x(t)} + \inprod{\nabla f(x(t))}{\dot x(t)}
		\\ &= \inprod{\dot x(t)}{\ddot x(t) + \nabla f(x(t))}
		\\ & = \inprod{\dot x(t)}{-\alpha \dot x(t)}
		\\ & = - \alpha \normsq{\dot x(t)}.
	\end{align}
	Integrating this equation gives
	\begin{align}
		\alpha
		\int_{0}^{t} \normsq{\dot x(s)} \dd s
		 & =
		\Phi(0) - \Phi(t) \\
		 & =
		f(x_0) - f(x(t)) - \frac{1}{2} \normsq{\dot x(t)}
		\label{eq:function-decrease-continuous}
		\\
		 & \le \Delta_f,
	\end{align}
	which completes the proof.
\end{proof}

Intuitively, \cref{lemma:gnorm-bar-x-continuous} implies that $\norm{\nabla f(\bar x(t))}$ is small if $\norm{\dot x(s)}$ is small for $s \in [0, t]$, whereas \cref{lemma:normsq_xdot_upperbound} states that the integral of $\normsq{\dot x(t)}$ is bounded by $\Delta_f/\alpha$.
By combining these two lemmas, we evaluate $\norm{\nabla f(\bar x(t))}$ and establish \cref{theorem:continuous-convergence-rate}.

\begin{proof}[Proof of \cref{theorem:continuous-convergence-rate}]
	Multiplying the inequality in \cref{lemma:gnorm-bar-x-continuous} by $(1 - e^{- \alpha t})^2$ yields
	\begin{align}
		\rbra*{1 - e^{- \alpha t}}^2
		\norm{\nabla f({\bar{x}(t)})}
		 & \le
		\alpha
		\rbra*{1 - e^{- \alpha t}}
		\norm{
			\dot x(t)
		} +
		\frac{L_2}{2}
		\int_{0}^{t} \normsq{\dot x(s)} e^{- \alpha (t - s)} \rbra*{t - s}
		\dd s  \\
		 & \le
		\alpha
		\norm{
			\dot x(t)
		} +
		\frac{L_2}{2}
		\int_{0}^{t} \normsq{\dot x(s)} e^{- \alpha (t - s)} \rbra*{t - s}
		\dd s.
		\label{eq:gnorm-bar-x-continuous_multiplied}
	\end{align}
	Since
	\begin{align}
		\int_0^T \rbra*{1 - e^{- \alpha t}}^2 \dd t
		=
		\sbra*{t + \frac{4 - e^{- \alpha t}}{2 \alpha e^{\alpha t}}}_{t=0}^T
		=
		T - \frac{3}{2 \alpha} + \frac{4 - e^{- \alpha T}}{2 \alpha e^{\alpha T}}
		\geq
		T - \frac{3}{2 \alpha},
	\end{align}
	integrating \cref{eq:gnorm-bar-x-continuous_multiplied} from $t = 0$ to $T$ yields
	\begin{align}
		\rbra*{T - \frac{3}{2 \alpha}}
		\min_{0 \le t \le T} \norm{\nabla f({\bar{x}(t)})}
		\leq
		\alpha
		\int_0^T \norm{\dot x(t)} \dd t
		+
		\frac{L_2}{2}
		\int_0^T
		\int_{0}^{t} \normsq{\dot x(s)} e^{- \alpha (t - s)} \rbra*{t - s}
		\dd s
		\dd t.
		\label{inequation:norm-upper-bound-with-two-integrals}
	\end{align}
	The first term on the right-hand side of \eqref{inequation:norm-upper-bound-with-two-integrals} is evaluated by using Cauchy--Schwarz inequality and \cref{lemma:normsq_xdot_upperbound} as follows:
	\begin{align}
		\alpha \int_{0}^{T} \norm{\dot x(t)} \dd t
		 & \le \alpha \sqrt{
			\rbra*{\int_{0}^{T} \dd t}
			\rbra*{\int_{0}^{T} \normsq{\dot x(t)} \dd t}
		}
		\\ & \le \sqrt{T\alpha {\Delta_f}}.
		\label{inequation:first-term-upper-bound}
	\end{align}
	The second term on the right-hand side of \cref{inequation:norm-upper-bound-with-two-integrals} is bounded as
	\begin{align}
		\int_0^T &
		\int_{0}^{t} \normsq{\dot x(s)} e^{- \alpha (t - s)} \rbra*{t - s}
		\dd s
		\dd t.
		\\ & = \int_{0}^{T} \rbra*{\int_{s}^{T}
			e^{-\alpha\rbra*{t - s}} \rbra*{t - s} \dd t} \normsq{\dot x(s)} \dd s
		\\ & = \int_{0}^{T} \rbra*{\int_{0}^{T-s} t e^{-\alpha t} \dd t} \normsq{\dot x(s)} \dd s
		\\ & = \int_{0}^{T}
		\rbra*{\sbra*{\rbra*{-\frac{t}{\alpha} - \frac{1}{\alpha^2}} e^{-\alpha t}}_{t=0}^{T-s}} \normsq{\dot x(s)} \dd s
		\\ & = \int_{0}^{T}
		\rbra*{\rbra*{-\frac{T-s}{\alpha} - \frac{1}{\alpha^2}} e^{-\alpha\rbra*{T-s}} - \rbra*{-\frac{1}{\alpha^2}}}  \normsq{\dot x(s)} \dd s
		\\ & = \int_{0}^{T}
		\rbra*{\frac{1}{\alpha^2} - \rbra*{\frac{T-s}{\alpha} + \frac{1}{\alpha^2}} e^{-\alpha\rbra*{T-s}}}  \normsq{\dot x(s)} \dd s
		\\ & \le \frac{1}{\alpha^2} \int_{0}^{T} \normsq{\dot x(s)} \dd s
		\\ & \hspace{60pt} \since{\rbra*{\frac{T-s}{\alpha} + \frac{1}{\alpha^2}} e^{-\alpha\rbra*{T-s}}\ge 0 \text{ on } s \in \cint{0}{T}}
		\\ & \le \frac{1}{\alpha^2} \frac{{\Delta_f}}{\alpha}
		\hspace{19pt}	\explain{\text{\cref{lemma:normsq_xdot_upperbound}}}
		\\ & = \frac{{\Delta_f}}{\alpha^3}.
		\label{inequation:second-term-upper-bound}
	\end{align}
	Applying \cref{inequation:first-term-upper-bound,inequation:second-term-upper-bound} to \cref{inequation:norm-upper-bound-with-two-integrals}, we have
	\begin{align}
		\rbra*{T - \frac{3}{2 \alpha}} \min_{0 \le t \le T} \norm{\nabla f({\bar{x}(t)})}
		& \le \sqrt{\alpha T{\Delta_f}} + \frac{L_2{\Delta_f}}{2\alpha^3}
		\label{eq:final_evaluation}
		\\ &  = \sqrt{\rbra*{\rbra*{3L_2}^{\frac{2}{7}} {\Delta_f}^{\frac{1}{7}} T^{-\frac{1}{7}}	} T{\Delta_f}} + \frac{L_2{\Delta_f}}{2\rbra*{\rbra*{3L_2}^{\frac{2}{7}} {\Delta_f}^{\frac{1}{7}} T^{-\frac{1}{7}}}^3}
		 \quad \explain{\cref{eq:def_alpha}}
		\\ &  = \frac{7 \cdot 3^{\frac{1}{7}}}{6} L_2^{\frac{1}{7}} {\Delta_f}^{\frac{4}{7}} T^{\frac{3}{7}}
		.
	\end{align}
	Moreover, using
	\begin{align}
		\rbra*{T - \frac{3}{2 \alpha}}^{-1}
		 & = T^{-1} \rbra*{1 - \frac{3}{2 \alpha T}}^{-1}
		\\ & = T^{-1}\rbra*{1 + O\rbra*{\alpha^{-1} T^{-1}}}
		\\ & = T^{-1} + O\rbra*{T^{-\frac{13}{7}}} \quad (T\rightarrow \infty) \quad \since{\alpha = \Theta\rbra*{T^{-\frac{1}{7}}}},
	\end{align}
	we have
	\begin{align}
		\min_{0 \le t \le T} \norm{\nabla f({\bar{x}(t)})}
		 & \le
		\rbra*{T - \frac{3}{2 \alpha}}^{-1}
		\frac{7 \cdot 3^{\frac{1}{7}}}{6} L_2^{\frac{1}{7}} {\Delta_f}^{\frac{4}{7}} T^{\frac{3}{7}}
		\\ & =
		\frac{7 \cdot 3^{\frac{1}{7}}}{6} L_2^{\frac{1}{7}} {\Delta_f}^{\frac{4}{7}} T^{\frac{3}{7}}
		\rbra*{T^{-1} + O\rbra*{T^{-\frac{13}{7}}}}
		\\ & =
		\frac{7 \cdot 3^{\frac{1}{7}}}{6} L_2^{\frac{1}{7}} {\Delta_f}^{\frac{4}{7}} T^{-\frac{4}{7}}
		+ O\rbra*{T^{-\frac{10}{7}}} \quad (T\rightarrow \infty).
		\label{inequation:min-norm-upper-bound-continuous}
	\end{align}
\end{proof}
The parameter $\alpha$ in \cref{eq:def_alpha} is chosen to minimize the right-hand side of \cref{eq:final_evaluation} in the above proof.

\section{Discussion and Future Work}
In this paper, we have shown that the convergence rate of the gradient norm at the average solution \cref{equation:bar-x-t-def} of ODE \cref{equation:ode} is $O(T^{-4/7})$, assuming the Lipschitz continuity of the gradient and Hessian of $f$.
This result directly implies that we have an $\epsilon$-stationary point of $f$ in $O(\epsilon^{-7/4})$ time.
Note that the parameter $\alpha$ in ODE~\eqref{equation:ode} depends on the final time $T$.
One direction for future work is to establish the same convergence rate of the gradient norm $\gnorm{\cdot}$ at the solution $x(t)$ of ODE~\eqref{equation:ode}, rather than at its time average $\bar{x}(t)$.

Furthermore, a natural next step is to discretize ODE~\eqref{equation:ode} and develop a first-order method that finds an $\epsilon$-stationary point of $f$ in $O(\epsilon^{-7/4})$ time.
However, care must be taken in discretizing the ODE.
\citet{goujaud2023provable} showed that, in the heavy-ball method
\begin{align}
	x_{k+1} = x_k - \gamma \nabla f(x_k) + \beta (x_k - x_{k-1}),
\end{align}
when the parameter $\beta$ is chosen close to $1$, the step size $\gamma$ must be taken very small to ensure the convergence of $(x_k)$ for any function $f$ with Lipschitz continuous gradient and Hessian.
At the same time, a naive discretization of ODE~\eqref{equation:ode} with our choice~\eqref{eq:def_alpha} of $\alpha$ yields a heavy-ball method with $\beta$ close to $1$.
This suggests that, to achieve a favorable convergence rate, it is likely to be necessary to discretize the ODE in a different way---potentially leading to a first-order method other than the heavy-ball method, such as Nesterov's accelerated gradient.

\section*{Acknowledgements}
This work was partially supported by JSPS KAKENHI (23H03351 and 24K23853) and JST CREST (JPMJCR24Q2).

\bibliography{ref.bib}

\end{document}